\documentclass{amsart}

\usepackage[cmtip,all]{xy}
\usepackage{amsmath, amsthm, latexsym, amssymb, amsfonts}
\usepackage{alltt}
\usepackage[pdftex]{graphicx,color} 
\usepackage{algorithm}

\usepackage{algpseudocode}
\usepackage{longtable}
\usepackage[table]{xcolor}
\usepackage{hyperref}
\usepackage{cleveref}

\newcommand{\supp}{\operatorname{supp}}

\newcommand{\dis}{\displaystyle}

\newcommand{\Hom}{\operatorname{Hom}}

\newcommand{\Ker}{\operatorname{Ker}} 
\newcommand{\im}{\operatorname{Im}}

\newcommand{\Cl}{\operatorname{Cl}}

\newcommand{\sig}{\sigma}
\newcommand{\Sig}{\Sigma}

\newcommand{\red}{\operatorname{red}}

\newcommand{\A}{\mathbb A}
\newcommand{\Q}{\mathbb Q}
\newcommand{\Z}{\mathbb Z}
\newcommand{\N}{\mathbb N}

\newcommand{\F}{\mathbb F}
\newcommand{\Pp}{\mathbb P}
\newcommand{\K}{{\mathbb K}}

\newcommand{\cC}{\cl C}

\newcommand{\cl}[1]{\mathcal{#1}}

\newcommand{\la}{\langle}
\newcommand{\ra}{\rangle}
\def\aa{{\bf \alpha}}
\def\bb{\beta}
\def\a{{\bf a}}
\def\t{{\bf t}}
\def\q{{\bf q}}

\def\uu{{\bf u}}
\def\vv{{\bf v}}
\def\x{{\bf x}}
\def\m{{\bf m}}

\def\kay{{\chi}}

\def\e{{\bf e}}
\def\ev{{\text{ev}}}

\newcommand{\vep}{\varepsilon}

\newtheorem{theorem}{Theorem}[section]
\newtheorem{lemma}[theorem]{Lemma}
\newtheorem{corollary}[theorem]{Corollary}

\newtheorem{proposition}[theorem]{Proposition}
\newtheorem{definition}[theorem]{Definition}

\newtheorem{example}[theorem]{Example}
\newtheorem{remark}[theorem]{Remark}

\numberwithin{equation}{section}

\begin{document}

\title[Vanishing Ideals for Toric Codes]{Computing Vanishing Ideals for Toric Codes}


\author[Mesut \c{S}ah\.{i}n]{Mesut \c{S}ah\.{i}n}
\address{ Department of Mathematics,
  Hacettepe  University,
  Ankara, TURKEY}
\curraddr{}
\email{mesut.sahin@hacettepe.edu.tr}
\thanks{The author is supported by T\"{U}B\.{I}TAK Project No:119F177}



\date{}


\begin{abstract}
Motivated by applications to the theory of error-correcting codes, we give methods for computing a generating set for the ideal generated by $\beta$-graded polynomials vanishing on certain subsets of a simplicial complete toric variety $X$ over a finite field $\mathbb{F}_q$, where $\beta$ is a $d\times r$ matrix whose columns generate a subsemigroup $\mathbb{N}\beta$ of $\mathbb{N}^d$. We also give a method for computing the vanishing ideal of the set of $\mathbb{F}_q$-rational points of $X$. When $\beta=[w_1 \cdots w_r]$ is a row matrix corresponding to a numerical semigroup $\mathbb{N}\beta=\langle w_1,\dots,w_r \rangle$, $X$ is a weighted projective space and generators of the relevant vanishing ideal is given using generators of defining (toric) ideals of numerical semigroup rings corresponding to semigroups generated by subsets of $\{w_1,\dots,w_r\}$.
\end{abstract}

\maketitle

\section{Introduction}
Let $\beta=[\bb_1 \cdots \bb_r]$ be a $d\times r$ matrix of rank $d$ with non-negative integer entries and $n=r-d>0$. The polynomial ring $S=\F[x_1,\dots,x_r]$ over a field $\F$ is made into a $\Z^d$-graded ring by letting $\deg_{\beta}(x_j):=\beta_j\in \N^d$, for $j\in [r]:=\{1,\dots,r\}$. Thus, $S=\bigoplus_{\aa \in \Z^d} S_{\aa}$, where $S_{\aa}$ is the \textit{finite-dimensional} vector space spanned by the monomials $\x^{\a}:=x_1^{a_1}\cdots x_r^{a_r}$ having degree $\aa=a_1\beta_1+\cdots+a_r\beta_r$ in the affine semigroup $\N\bb$ by \cite[Theorem 8.6]{CombComAlgBook}. This leads to the following short exact sequence 
\begin{equation} \label{e:ses1}
 \dis \xymatrix{ 0  \ar[r] & \Z^n \ar[r]^{\phi} & \Z^r \ar[r]^{{\bb}} & \Z^d \ar[r]& 0},
\end{equation}
where $\phi$  denotes a matrix such that $\im (\phi) = \Ker (\bb)$. Applying $\Hom(-,\K^*)$ for an algebraically closed field $\K$, we get the dual short exact sequence

\begin{equation} \label{e:ses2}
\dis \xymatrix{ 1  \ar[r] & (\K^*)^d \ar[r]^{i} & (\K^*)^r \ar[r]^{\pi} & (\K^*)^n \ar[r]& 1},
\end{equation}
where $\pi:(t_1,\dots, t_r)\mapsto (\mathbf{t}^{\uu_1}, \dots , \mathbf{t}^{\uu_n}),$ with $\uu_1,\dots, \uu_n$ being the columns of $\phi$. Denote by $G=\Ker(\pi)\cong (\K^*)^d$. Then, $G$ is an algebraic subgroup of $(\K^*)^r$ acting on the affine space $\A^r$ over $\K$ by coordinate-wise multiplication. We denote by $\A^r_{G}$ the set $\K^r/G$ of $G$-orbits. More generally, $Y_G$ denotes the set $Y/G$ of $G$-orbits of elements in $Y\subseteq \A^r$. In general, $\A^r_{G}$ is not necessarily a variety, but Geometric Invariant Theory (GIT, for short) says removing some \textit{bad} orbits we can get nice quotient spaces which are varieties. Toric varieties are such important nice quotient spaces lying at the crossroad of combinatorics, commutative algebra and algebraic geometry with numerous applications to areas such as biology, chemistry, coding theory, physics and statistics.

The algebraic set up above arise often within toric geometry which we briefly explain now. When $X$ is an $n$-dimensional simplicial complete toric variety over a field, the first map in equation (\ref{e:ses1}) is just multiplication by the matrix $\phi$ whose rows are the primitive generators $\vv_1,\dots,\vv_r\in \Z^n$ of the rays in the corresponding fan. Under suitable conditions, satisfied by smooth varieties for instance, the variety $X$ can be represented as a GIT quotient, i.e. $X\cong (\K^r \setminus V(B))/G$, where $B$ is a monomial ideal of $S$ determined by the cones in the fan, see Section \ref{s:toricVariety} for details. 

In applications to coding theory, we work with a finite field $\F=\F_q$ together with an algebraic closure $\K=\overline{\F}_q$ and identify $\F_q$-rational points $\A^r_G(\F_q)$ of $\A^r_G$ with $\F_q^r/G$,  where $G=\{ \t \in (\overline{\F}^*_q)^r : \t^{\uu_1}= \cdots =\t^{\uu_n}=1\}$ is the algebraic group determined in equation (\ref{e:ses2}) for $\K$. Therefore, $\F_q$-rational points $X(\F_q)$ of $X$ is identified with the set of orbits $(\F_q^r \setminus V(B))/G=\A^r_G(\F_q) \setminus V_G(B)$.

Toric codes, considered for the first time by Hansen \cite{HansenAAECC}, can be obtained by evaluating all homogeneous polynomials in the space $S_{\alpha}$ at only the $\F_q$-rational points of the dense torus $T_X \subset X$. They are studied intensively from different points of view, see \cite{JoynerAAECC,LS06SIAMJDM,Ruano07FFA,SoSo08SIAMJDM,Ruano09JSC,SoSo10SIAMJDM,JA11FFA,BK13JSC,LittleFFA2013,Sop13JSC,CelSop15FFA,MultHFuncToricCodes16}. A row of a \textit{generator matrix} of the code is obtained by evaluating a monomial in a basis of $S_{\alpha}$ at the $\F_q$-rational points so that the code is the row space of the matrix, for sufficiently large $q$. Some record breaking examples are found replacing the vector space $S_{\alpha}$ by its subspaces, see \cite{7championCodes} and references therein. The latter corresponds to deleting rows from a generating matrix of the toric code, which is investigated by Little \cite{ToricsFiniteGeom} using the theory of finite geometries. See also Hirschfeld \cite{GlynnHirschfeld} for another example relating finite geometry and vanishing ideals.

One can also add/delete columns to/from a generating matrix in order to get a better code, which correspond to considering a proper subset/superset of $T_X$. In this regard, Nardi offered to extend the length of a toric code by evaluating at the full set of $\F_q$-rational points $X(\F_q)$ in \cite{NardiHirzebruch} and \cite{NardiProjToric}. There is yet another extension of classical toric codes, which we introduce now. As in the toric case, we evaluate polynomial functions from $S_{\aa}:=\F_q[x_1,\dots,x_r]_{\aa}$ at the $\F_q$-rational points $[P_1],\dots,[P_N]$ of a subset $Y_G \subseteq \A^r_G(\F_q)$, defining the following $\F_q$-linear map $$\ev_{Y_G}:S_\aa\to \F_q^N,\quad F\mapsto (F(P_1),\dots,F(P_N)).$$  
The image $\text{ev}_{Y_G}(S_\aa) \subseteq \F_q^N$ denoted by ${\cC}_{\aa,Y_G}$ is called an \textbf{evaluation code on orbits}. The main three parameters $[N,K,\delta]$ of these codes are the \textit{length} $N$ of ${\cC}_{\aa,Y_G}$ which is the size $|Y_G|$, the \textit{dimension} $K=\dim_{\F_q}({\cC}_{\aa,Y_G})$ of the image as a subspace of $\F_q^N$, and the \textit{minimum distance} $\delta$ which is the smallest \textit{weight} among all code words $c\in{{\cC}_{\aa,Y_G}}\setminus\{0\}$, where the weight of $c$ is the number of non-zero components. Since the kernel of the map $\ev_{Y_G}$ is nothing but $I_{\aa}(Y_G):=I(Y_G)\cap S_{\aa}$, the code ${\cC}_{\aa,Y_G}$ is isomorphic to $S_{\aa}/I_{\aa}(Y_G)=(S/I(Y_G))_{\aa}$. Hence, computing a minimal generating set for the vanishing ideal $I(Y_G)$ is of central importance. When $X\subset Y_G \subseteq \A^r_G(\F_q)$, the new codes are lengthier and one has the chance to choose the subset $Y_G$ so that the other parameters improves as well, see Example \ref{eg:goodcode}. As pointed out in \cite{NardiProjToric}, as the length increases one can build secret sharing schemes based on these codes with more participants, see \cite{HansenSSSWSM}.

In the present paper, we start by observing that $I(\A^r_G(\F_q))$ has a minimal generating set consisting of binomials. We also give a conceptual method to list binomial generators for $I(\A^r_G(\F_q))$ using the cell decomposition of the affine space $\A^r$, see Theorem \ref{t:idealAffine}. The vanishing ideal of the $\F_q$-rational points of the toric variety $X$ can be obtained as a colon ideal of $I(\A^r_G(\F_q))$ with respect to the monomial ideal $B$, see Theorem \ref{t:saturation}. As applications, we give three binomials generating $I(\A^4_G(\F_q))$ and thereby obtain a binomial and a polynomial with $4$ terms generating $I(X(\F_q))$ minimally, where $X=\cl H_{\ell}$ with $\ell>1$ is the Hirzebruch surface, see Theorem \ref{t:I(A4)Hirzebruch} and Theorem \ref{t:IdealHirzebruch}, revealing that $X(\F_q)$ is an ideal theoretic complete intersection. It is known that $I(Y_G)$ is a binomial ideal when $Y_G$ is a submonoid of $\A^r_G$, \cite[Proposition 2.6]{sahin18} whereas $I(X(\F_q))$ can still be binomial even if $X$ is not a monoid, see Theorem \ref{t:IdealWPS}. The last theorem generalizes to some weighted projective spaces the fact that the ideal $I(\Pp^n(\F_q))$ has binomial generators given explicitly by Mercier and Rolland
\cite{MRidealOfPn}. It is worth pointing out that these binomials form a Groebner basis as shown by Beelen, Datta and Ghorpade \cite{BDGidealOfPn} which is used to obtain a footprint bound for the minimum distance of the corresponding code. Binomial ideals appear as vanishing ideals in many works, see e.g. \cite{vPV13CA,NvP14JPAA,NevesJAA20,NvPV20JA,BaranSahin} and prove useful in studying basic parameters of the related codes. As a last application, see Theorem \ref{t:dimension}, we use binomiality of the vanishing ideal $I(\A^r_G)$ to give another proof for a very useful combinatorial method established for the first time by Nardi in \cite{NardiHirzebruch,NardiProjToric} to compute dimension of a code obtained on $X(\F_q)$.

\section{Binomial Vanishing Ideals}
In this section, we list some basic cases where the \textit{homogeneous} or \textit{multigraded} vanishing ideal $I(Y_G)$ of a subset $Y_G$ is generated by binomials. Recall that the set of all polynomials vanishing on the subset $Y$ is an ideal called the vanishing ideal of $Y$ which differs from the $\bb$-graded vanishing ideal $I(Y_G)$ of $Y_G:=Y/G$ that is generated by \textit{homogeneous} (or $\bb$-graded) polynomials vanishing on $Y$. 

Binomial ideals play a central role at the crossroad of combinatorics, commutative algebra, convex and algebraic geometry, see the recent book \cite{HHObookBinomialIdeals} by Herzog, Hibi and Ohsugi for a through introduction to their theory and applications. It is an emerging hot topic relating as diverse areas as commutative algebra, graph theory, coding theory  and statistics. They have many interesting properties discovered starting from the seminal work \cite{BinomialIdeals} by Eisenbud and Sturmfels, and their decompositions are studied further by other authors, see e.g.   \cite{OjedaCellularBinomials,EserMatusevichJLMS}. There is a \verb|Macaulay 2| package \cite{Kahle12DecBinMac2} for their binomial primary decomposition as well.

Notice that if $F\in S_{\aa}$ then we have 
\begin{equation}\label{e:indepenceOfTheIdealFromG}
F(g\cdot P)=g^{\aa}F(P)=0 \text{ if and only if } F(P)=0, \text{ for any } g\in G.
\end{equation}

\begin{remark} \label{r:indepenceOfTheIdealFromK} Recall that $G$ is defined over the field $\K=\overline{\F}_q$ in applications to coding theory where we also take $\F=\F_q$. But the vanishing of a polynomial at a point $[P]$ is independent of the group $G$ by Equation \ref{e:indepenceOfTheIdealFromG}.  Therefore,  the homogeneous generators for the vanishing ideal $I(Y_G)\subseteq \F[x_1,\dots,x_r]$ would be the same even if $\K= \F_q$.
\end{remark}

\begin{remark} When $Y_G$ is the subgroup $Y_Q$ of the torus $(\F^*_q)^r/G$ parameterized by a matrix $Q=[\q_1 \cdots \q_r]$ its vanishing ideal is proven in \cite{AlgebraicMethods2011,sahin18} to be some special \textbf{binomial ideal} known as a \textbf{lattice ideal}, i.e., it is of the form 
\[I_L:=\langle \x^{\m^{+}}- \x^{\m^{-}} \, : \m=\m^{+}-\m^{-} \in L \, \rangle
\]
for a lattice (finitely generated abelian group) $L$, where $\m^{+}$ and $\m^{-}$ record positive and negative components of $\m$.
\end{remark}

Clearly, $\A^r_{G}(\F_q)$ is a monoid under coordinatewise multiplication with identity element $(1,\dots,1)$. The vanishing ideal of a submonoid is known to be binomial:
\begin{proposition} \label{prop:binomial} \cite[Proposition 2.6]{sahin18} If $Y_G$ is a submonoid of $\A^r_{G}(\F_q)$, then $I(Y_G)$ is binomial. 
\end{proposition}

\begin{corollary} The ideal $I(\A^r_{G}(\F_q))$ is binomial.
\end{corollary}
\begin{proof} The proof follows from \Cref{prop:binomial} by taking $Y_G=\A^r_{G}(\F_q)$.
\end{proof}

\section{Cellular Binomial Ideals for Orbits}

In this section, we see that the vanishing ideals of points and of orbits are special binomial ideals. Throughout the section, we assume that both fields $\F=\K=\F_q$ in the virtue of Remark \ref{r:indepenceOfTheIdealFromK}. Let us start by explaining what we mean from special in this regard: 

\begin{definition}\cite[Definition 2.2]{EserMatusevichJLMS} An ideal $J \subseteq \F[x_1,\dots,x_r]$ is cellular if every variable $x_j$ is either a nonzerodivisor or nilpotent modulo $J$. If $J$ is a cellular binomial ideal, and $\emptyset \neq \varepsilon \subseteq [r]$ indexes the variables that are nonzerodivisor modulo $J$, then $J$ is called $\varepsilon$-cellular.
\end{definition} 

\begin{definition}
Let $S=\F[x_1,\dots,x_r]$ be a polynomial ring and $\emptyset \neq \varepsilon \subseteq [r]$. $S[\varepsilon]$ denotes the ring $\F[x_i : i \in \varepsilon]$ and we define $\mathfrak{m} (\check{\varepsilon}):=\langle x_i : i\notin \varepsilon \rangle \subseteq S$.
\end{definition}

\begin{definition}
The support $\varepsilon_p$ of a point $P\in \A^r$, is the set of indices $i\in [r]$ for which the $i-th$ component $p_i$ of $P$ is not zero. So, $\A^r$ is the disjoint union of its subsets $\A^r(\varepsilon) $ consisting of the points supported at $\varepsilon\subseteq [r]$. Notice that $\A^r(\emptyset)=\{(0,\dots,0)\}$ and $\A^r([r])=(\K^*)^r $.
\end{definition}

We consider the projection $\pi_{\varepsilon}:\A^r\rightarrow \A^{|\varepsilon|}$ where $\pi_{\varepsilon}(x_1,\dots,x_r)=(x_{i_1},\dots,x_{i_k})$ for any subset  $\varepsilon=\{i_1,\dots,i_k\}\subseteq [r]$. By abusing the notation, we use the same notation for the homomorphism $\pi_{\varepsilon}:\Z^r\rightarrow\Z^{|\varepsilon|}$.  \\

We distinguish $L_{\bb}({\varepsilon})=\{(m_1,\dots,m_r)\in L_{\beta}: m_i=0, \forall i \notin \varepsilon\}$  with its image $\pi_{\varepsilon}(L_{\bb}({\varepsilon}))$ under $\pi_{\varepsilon}:\Z^r\rightarrow\Z^{\varepsilon}$. Note that $$(m_1,\dots,m_r)\in L_{\beta} \iff m_1\beta_1+\dots+m_r\beta_r=0.$$ 
Thus, $$\m\in L_{\bb}({\varepsilon}) \iff \sum\limits_{i\in \varepsilon}m_i\beta_i=0\iff \pi_{\varepsilon}(\m)\in L_{\bb(\varepsilon)}:=\mbox{ Ker }(\beta({\varepsilon})),$$
where $\beta({\varepsilon})$ is the matrix with columns $\bb_j$ for $j\in \varepsilon$. Thus, $\pi_{\varepsilon}(L_{\bb}({\varepsilon}))=L_{\bb(\varepsilon)}$. \\

Recall that $\kay_{p} : L_{\bb(\varepsilon)} \rightarrow \K^*$ is defined by $\kay_p(\m)=\x^{\m}(P)$, and the ideal $I_{\kay_{p},L_{\bb(\varepsilon)}}$ is generated by binomials of the form $\x^{\m^{+}}-\x^{\m}(P)\x^{\m^{-}}$ for $\m\in L_{\bb(\varepsilon)}$. 

Our first $\varepsilon$-cellular binomial ideals appears here:
\begin{proposition}\label{p:IdealPoint}
With the notations above and $[P]:=G\cdot P$, we have the following
\begin{enumerate}
\item $I([1_{\varepsilon}])=\mathfrak{m} (\check{\varepsilon}) +S\cdot I_{L_{\bb(\varepsilon)}}$, where $1_{\varepsilon}\in \A^r(\varepsilon)$ is the point whose image $\pi_{\varepsilon}(1_{\varepsilon})=(1,\dots,1)\in \A^{|\varepsilon|}$.
\item $I([P])=\mathfrak{m} (\check{\varepsilon})+S\cdot I_{\kay_p,L_{\bb(\varepsilon)}}$, where $\varepsilon$ is the support of $P\in \A^r$ and $I_{\kay_{p},L_{\bb(\varepsilon)}}$ is the lattice ideal of the partial character $\kay_{p}$.
\end{enumerate}
\end{proposition} 

\begin{proof}
\begin{enumerate}
    \item Clearly, $x_i$ vanishes at $1_{\varepsilon}$ when $i\notin\varepsilon$. So, $\mathfrak{m} (\check{\varepsilon})=\langle x_i : {i\notin\varepsilon}\rangle\subseteq I([1_{\varepsilon}])$. Obviously, the homogeneous binomial $\x^{\m^{+}}-\x^{\m^{-}}\in I_{L_{\bb(\varepsilon)}}\subset S[{\varepsilon}],$ vanishes at $1_{\varepsilon}$, as $1-1=0$. Therefore, $I_{L_{\bb(\varepsilon)}}\subseteq I([1_{\varepsilon}])$ proving the first containment. 
    
     Now, let $F \in I([1_{\varepsilon}])$ be a homogeneous generator with monomials not contained in $\mathfrak{m} (\check{\varepsilon})$. It follows from \Cref{prop:binomial} that $I([1_{\varepsilon}])\cap S[\varepsilon]$ is a binomial ideal. So, $F=c_1\x^{\a_1}+c_2\x^{\a_2}$. Thus, $c_1+c_2=F(1_{\varepsilon})=0$ implying $F=c_1(\x^{\a_1}-\x^{\a_2})$. As $F$ is a homogeneous polynomial supported at $\varepsilon$, we have $\a_1-\a_2\in L_{\bb(\varepsilon)}.$ Thus, $F\in I_{L_{\bb(\varepsilon)}}.$
    
    \item $\mathfrak{m} (\check{\varepsilon})=\langle x_i : {i\notin\varepsilon}\rangle\subseteq I([P])$ follows from the assumption that $\varepsilon$ is the support of $P\in \A^r$. Let $F\in I([P]) \setminus \mathfrak{m} (\check{\varepsilon})$. We proceed as in the proof of \cite[Theorem 5.1]{sahin18}. Then $F\in S[\varepsilon]$ and $F(P)=0 \iff F'(1_{\varepsilon})=0$, for $F'(x_{i_1},\dots,x_{i_k})=F(p_{i_1}x_{i_1},\dots,p_{i_k}x_{i_k})$ when $\varepsilon=\{i_1,\dots,i_k\}$. Since, the polynomial $F'\in I_{L_{\bb(\varepsilon)}}$ is an algebraic combination of binomials  $\x^{\m^{+}}-\x^{\m^{-}}$ for the elements $\m\in L_{\bb(\varepsilon)}$, it follows that $F\in I_{\kay_{p},L_{\bb(\varepsilon)}}$, as 
    $$(\x^{\m^{+}}-\x^{\m}(P)\x^{\m^{-}})(P)=0\iff (\x^{\m^{+}}-\x^{\m^{-}})(1_{\varepsilon})=0.$$
\end{enumerate}
These complete the proof.
\end{proof}
Let $T=\{(t_1,\dots,t_r)\in \A^r: t_1 \cdots t_r\neq0\}$ be the torus $(\K^*)^r$ of $\A^r$ and let $T_{G}$ denote the quotient group $T/ G$. Then $T_G$ acts on $\A^r_G$ via coordinate wise multiplication:
$$ \displaystyle
T_{G}\times \A^r_G  \rightarrow \A^r_G, \quad 
([\t],[P]) \rightarrow [\t P].
$$
It is easy to see that $\A^r_G(\varepsilon)=T_{G}\cdot [1_{\varepsilon}]\cong (\K^*)^{|\varepsilon|}$, since for every $P\in \A^r(\varepsilon)$, there is a unique $\t\in T$ with $P=\t \cdot 1_{\varepsilon}$, where $t_j=p_j$ when $j\in \varepsilon$ and $t_j=1$ when $j\notin \varepsilon$.

Next, we show that the vanishing ideals of orbits (of cells) are $\varepsilon$-cellular binomial.
\begin{theorem} \label{t:idealOrbit} With the notations above and $\K=\F_q$ we get the following result,
$$I(\A^r_G(\varepsilon))=I(T_{G}\cdot [1_{\varepsilon}])=\mathfrak{m} (\check{\varepsilon})+S\cdot I_{(q-1)L_{\bb(\varepsilon)}}.$$
\end{theorem}
\begin{proof} A polynomial $F\in I(T_{G}\cdot [1_{\varepsilon}])$ with monomials not contained in $\mathfrak{m} (\check{\varepsilon})$ lie in $S[\varepsilon]$ so that $F\in S[\varepsilon]\cap I([T_{G}\cdot 1_{\varepsilon}])=I(G_{\varepsilon}\cdot T_{\varepsilon})$, where $G_{\varepsilon}=\pi_{\varepsilon}(G)$ and $T_{\varepsilon}=\pi_{\varepsilon}(T)=(\K^*)^{|\varepsilon|}$.  By \cite[Corollary 4.14]{sahin18}, we have $I(T_{G})=I_{(q-1)L_{\beta}}$ which corresponds to the case where $\varepsilon=[r]$. We can prove similarly that  $I(G_{\varepsilon}\cdot T_{\varepsilon})=I_{(q-1)L_{\bb(\varepsilon)}}$, for the other $\emptyset \neq \varepsilon \subset [r]$. Therefore, $F \in S\cdot I_{(q-1)L_{\bb(\varepsilon)}}$.
\end{proof}
\begin{corollary}
$\displaystyle I(\A^r_G)=\bigcap_{\varepsilon\subseteq[r]}I(T_{G}\cdot [1_{\varepsilon}])$.
\end{corollary}
\begin{proof} 
Follows from $\displaystyle \A_{G}^r=\bigcup_{\varepsilon\subseteq[r]} \A^r_G(\varepsilon)=\bigcup_{\varepsilon\subseteq[r]}T_{G}\cdot [1_{\varepsilon}]$.
\end{proof}

\begin{theorem}\label{t:idealAffine} Let $\x^{\varepsilon}:=\prod_{i\in \varepsilon }^{} x_i=x_{i_1}\cdots x_{i_k}$ for $\varepsilon=\{i_1,\dots,i_k\}$. Then,
$$I(\A_{G}^r)=\sum\limits_{\emptyset \neq \varepsilon\subseteq[r]} \x^{\varepsilon} \cdot I_{(q-1)L_{\bb(\varepsilon)}}.$$
\end{theorem}
\begin{proof}
Firstly, we show that $I(\A_{G}^r)\subseteq \sum\limits_{\varepsilon\subseteq[r]}\x^{\varepsilon}\cdot I_{(q-1)L_{\bb(\varepsilon)}}$. We know that $I(\A_{G}^r)$ is pure binomial. So, its generators are of the form
$\x^\a(\x^{\m^+}-\x^{\m^-})\in I(\A_{G}^r)$. Then we claim that $\mbox{supp}(\x^{\m^+})\cup \mbox{supp}(\x^{\m^-})\subseteq\varepsilon$ for $\varepsilon=\mbox{supp}(\x^\a)\subseteq[r]$. If not, say there exists $i\in \mbox{supp}(\x^{\m^+})\backslash \varepsilon$, then consider the point $P$ whose $i$-th coordinate is $0$ and others are $1$. Then $\x^\a(\x^{\m^+}-\x^{\m^-})(P)=-1\neq 0$. The other option leads to a contradiction, similarly. Since $\displaystyle \A_{G}^r=\bigcup_{\varepsilon\subseteq[r]} \A^r_G(\varepsilon)=\bigcup_{\varepsilon\subseteq[r]}T_{G}\cdot [1_{\varepsilon}]$, it follows that $I(\A_{G}^r)\subseteq I(T_{G}\cdot [1_{\varepsilon}])$. So, $\x^{\m^+}-\x^{\m^-}\in I(T_{G}\cdot [1_{\varepsilon}])$, since $\x^\a\neq 0$ on $T_{G}\cdot [1_{\varepsilon}]$. Thus, $\x^{\m^+}-\x^{\m^-}\in S[\varepsilon]\cap I(T_{G}\cdot [1_{\varepsilon}])=I_{(q-1)L_{\bb(\varepsilon)}}$.\\

For the other direction, take $F\in I_{(q-1)L_{\bb(\varepsilon)}}=S[\varepsilon]\cap I(T_{G}\cdot [1_{\varepsilon}])$. Then the polynomial $\x^{\varepsilon}F=x_{i_1}\cdots x_{i_k}F$ vanishes on $\A_{G}^r$ as we explain next. Every $[P]\in \A_{G}^r$ lies in an orbit $T_{G}\cdot 1_{\varepsilon_p}$ for some $\varepsilon_p\subseteq [r]$. If $i\in \varepsilon\backslash\varepsilon_p\neq \emptyset$ then $p_i=0$ so that $\x^{\varepsilon}(P)=0$. Otherwise, $\varepsilon\subseteq\varepsilon_p$ and $\x^{\varepsilon}(P)\neq 0$. Introduce a new point $P'\in \A^r(\varepsilon)$ whose $i$-th coordinate coincides with that of $P$, i.e. $p_i=p'_i$ for all $i\in \varepsilon$. As $F\in S[\varepsilon]$, we have $F(P)=F(p_{i_1},\dots,p_{i_k})=F(P')=0$ for $[P']\in T_{G}\cdot 1_{\varepsilon}$. Therefore, we have $\x^{\varepsilon}F(P)=0$ in any case, completing the proof. 
\end{proof}
\section{Vanishing Ideal of Rational Points of a Toric Variety} \label{s:toricVariety}
Let $X=X_{\Sigma}$ be a simplicial complete toric variety over an algebraically closed field $\K=\overline{\F}_q$. Then, by a celebrated result due to Cox (see \cite{CoxRingToric}), the $\K$-rational points $X(\K)$ of the toric variety $X$, is isomorphic to the geometric invariant theory quotient $(\K^r\setminus V(B)) /G$, for the monomial ideal 
$$\displaystyle B=\langle  \x^{\hat{\sig}}=\prod_{\rho_i \notin \sig}^{} \: x_i  : \sig \in \Sig \rangle \subset S=\F_q[x_1,\dots,x_r] \text{ and }$$
 \begin{align*}
G=V(I_{L_{\bb}})\cap (\K^*)^r:&=\{P \in (\K^*)^r \: | \: (\x^{\m^+}- \x^{\m^-})(P)=0 \:\mbox{for all} \: \m\in L_{\bb} \}\\
&=\{P \in (\K^*)^r \: | \: \x^{\m}(P)=1 \:\mbox{for all} \: \m\in L_{\bb} \}.    
\end{align*}
Therefore, $\K$-rational points of $X$ are in bijection with the orbits $[P]:=G\cdot P$, for $P\in \K^r\setminus V(B)$. Hence, we may regard them as elements of the set $\A^r_G \setminus V_G(B)$. It follows that the $\F_q$-rational points of $X$ are in bijection with the orbits $[P]:=G\cdot P$, for $P\in \F_q^r\setminus V(B)$.

\begin{theorem}\label{t:saturation} If $Y\subseteq \A^r$, then the vanishing ideal in $S$ of the subset $[Y\backslash V(B)]$ of $\A^r_G \setminus V_G(B)$ is given by
$I([Y\backslash V(B)])=I(Y_G):B$.
\end{theorem}
\begin{proof}
As $V(B)$ is $G$-invariant we first notice that
$$[Y\backslash V(B)]=[Y]\backslash [V(B)]:=Y_G \setminus V_G(B).$$

First we prove the inclusion $I([Y\backslash V(B)])\subseteq I(Y_G):B$. Let $F\in I([Y\backslash V(B)])$ be a homogeneous polynomial. Then $F$ vanishes on $Y\backslash V(B)$. Since $F'$ vanishes on $V(B)$, for all $F'\in B$, $FF'$ vanishes on $Y$. For $F'=\bigoplus_{\alpha \in \N\beta}F'_{\alpha}\in B$, we have $F'_{\alpha}\in B$, $\forall \alpha \in \N\beta$ as $B$ is a homogeneous ideal. So, $FF'_{\alpha}$ is a homogeneous polynomial vanishing on $Y_G$, i.e. $FF'_{\alpha}\in I(Y_G)$ is a homogeneous generator, and hence $FF'\in I(Y_G)$. Thus, $F\in I(Y_G):B$.

Now we show the other containment. As $I(Y_G):B$ is homogeneous, we start by taking a homogeneous generator $F$ of $I(Y_G):B$. Then $FF'\in I(Y_G)$, $\forall F' \in B$. Let us take $P\in Y\backslash V(B)$. Since $P\notin V(B)$, there is a polynomial $F'\in B$ such that $F'(P)\neq 0$. As $P\in Y$, we have $F(P)F'(P)=0,$ so $F(P)=0$. Therefore, $F\in I([Y\backslash V(B)])$.
\end{proof}

\begin{corollary} \label{c:saturation}
$I(X(\F_q))=I(\A_{G}^r(\F_q)):B$
\end{corollary}
\begin{proof} Follows from Theorem \ref{t:saturation} by taking $Y=\A^r(\F_q)$. 
\end{proof}

\section{Applications}
In this section, we compute vanishing ideals of $\F_q$-rational points of some famous examples of toric varieties applying the theory developed in previous sections.

\subsection{Hirzebruch Surfaces}
\label{ex:Hirzebruch} 	
Let $X=\cl H_{\ell}$ be the Hirzebruch surface whose primitive ray generators are as follows $\vv_1=(1,0)$, $\vv_2=(0,1)$, $\vv_3=(-1,\ell)$,
and $\vv_4=(0,-1)$, for any positive integer $\ell$. The exact sequence becomes
$$\dis \xymatrix{ \mathfrak{P}: 0  \ar[r] & \Z^2 \ar[r]^{\phi} & \Z^4 \ar[r]^{\beta}& \Cl(\cl H_{\ell}) \ar[r]& 0},$$
for $\phi=[\uu_1 \: \: \uu_2]$ with $\uu_1=(1,0,-1,0)$, $\uu_2=(0,1,\ell,-1)$ and $\beta=\begin{bmatrix}
1 & 0 & 1& \ell\\
0 & 1 & 0& 1  
\end{bmatrix}$ with $L_{\beta}=\la \uu_1, \uu_2\ra$.
The dual sequence over $\K=\overline{\F}_q$ is
$$\dis \xymatrix{ \mathfrak{P}^*: 1  \ar[r] & \mathcal{G} \ar[r]^{i} & (\K^*)^{4} \ar[r]^{\pi} & (\K^*)^2 \ar[r]& 1}$$
where $\pi:\t\mapsto (t_1t_3^{-1},t_2t_3^{\ell}t_4^{-1}).$ 

Then the class group is $\Cl(\cl H_{\ell})\cong \Z^2$ and the group acting on the affine space is
$$G=\Ker(\pi)=\{(t_1,t_2,t_1,t_1^{\ell}t_2)\;|\;t_1,t_2\in\K^*\}\cong (\K^*)^2.$$ Hence, $\K$-rational points of the torus is $ T_X(\K) \cong {(\K^*)^2}\cong(\K^*)^{4} /G$ whereas $\F_q$-rational points is $ T_X(\F_q) \cong {(\F_q^*)^2}\cong(\F_q^*)^{4} /G$. Indeed, $\F_q$-rational points of $X$ is given by $X(\F_q) \cong(\F_q^4\setminus V(B)) /G=\A^4_G(\F_q)\setminus V_G(B)$, where  $$B=\langle x_1,x_3 \rangle \cap \langle x_2,x_4 \rangle=\langle x_1x_2,x_1x_4,x_3x_2,x_3x_4 \rangle$$ 
is the irrelevant ideal in the Cox ring $S=\F_q[x_1,x_2,x_3,x_4]$ which is $\Z^2$-graded via  $$\deg_{\bb}(x_1)=\deg_{\bb}(x_3)=(1,0),\quad \deg_{\bb}(x_2)=(0,1), \quad \deg_{\bb}(x_4)=(\ell,1).$$ 
The vanishing ideal of $\A^4_G(\F_q)$ over the field $\F_q$ is given below.
\begin{theorem} \label{t:I(A4)Hirzebruch}
$I(\A^4_G(\F_q))=\langle {x}_{3}{x}_{1}f_1, \quad {x}_{4}{x}_{2}{x}_{1}f_2, \quad  {x}_{4}{x}_{3}{x}_{2}f_3\rangle$, where
$$\begin{array}{lllll} 
&f_1={x}_{3}^{q-1}-{x}_{1}^{q-1}, \,  &f_2={x}_{4}^{q-1}-{x}_{2}^{q-1}{x}_{1}^{(q-1)\ell} \,
\text{ and } &f_3={x}_{4}^{q-1}-{x}_{3}^{(q-1)\ell}{x}_{2}^{q-1}.
\end{array}$$
\end{theorem}
\begin{proof}

Recall that $\varepsilon \subseteq [4]$ gives the matrix $\bb(\varepsilon)$ with columns $\bb_j$ for $j\in \varepsilon$. For instance, if $\varepsilon=\{1,2,4\}$, then $\bb(\varepsilon)=\begin{bmatrix}
1 & 0 & \ell\\
0 & 1 & 1  
\end{bmatrix}$ whose kernel is as follows 
$$L_{\bb(\varepsilon)}=\{(a_1,a_2,a_4)\in \Z^3 : a_1+\ell a_4=a_2+a_4=0\}=\{(-\ell a_4,-a_4,a_4) : a_4 \in \Z\}.$$ 
Thus, the corresponding toric ideal is $I_{L_{\bb(\varepsilon)}}=\langle x_4-x_2x_1^{\ell} \rangle$. Similarly, for 
$$\begin{array}{lllll} 
&\varepsilon=\{2,3,4\}, \text{ we have } I_{L_{\bb(\varepsilon)}}=\langle x_4-x_3^{\ell}x_2 \rangle, \\  
&\varepsilon=\{1,2,3\}, \varepsilon=\{1,3,4\} \text{ or } \varepsilon=\{1,3\} \text{ we have } I_{L_{\bb(\varepsilon)}}=\langle x_3-x_1 \rangle, \\
&\varepsilon=\{1,2,3,4\} \text{ we have } I_{L_{\bb(\varepsilon)}}=\langle x_3-x_1, x_4-x_2x_1^{\ell}\rangle=\langle x_3-x_1, x_4-x_3^{\ell}x_2 \rangle.
\end{array}$$
For any other $\varepsilon \subseteq [4]$, the kernel $L_{\bb(\varepsilon)}$ is trivial and so is the toric ideal. By Theorem \ref{t:idealAffine}, the ideal $I(\A^4_G)(\F_q)$ is generated by $\x^{\varepsilon}I_{(q-1)L_{\bb(\varepsilon)}}$, so it is generated by the following binomials:
$$\begin{array}{lllll} 
&x_1x_2x_4f_2 &\text{ for } \varepsilon=\{1,2,4\}, \\
&x_2x_3x_4f_3 &\text{ for } \varepsilon=\{2,3,4\}, \\  
&x_1x_2x_3f_1 &\text{ for } \varepsilon=\{1,2,3\}, \\
&x_1x_3x_4f_1 &\text{ for } \varepsilon=\{1,3,4\}  \\
&x_1x_3f_1  &\text{ for } \varepsilon=\{1,3\}, \\
&x_1x_2x_3x_4f_1,x_1x_2x_3x_4f_2 (\text{ or } x_1x_2x_3x_4f_3) &\text{ for } \varepsilon=\{1,2,3,4\}.
\end{array}$$
As some binomials divide other, the proof follows.
\end{proof}

We will use the following algorithm to compute generators of the intersections of ideals, that is given right after Theorem 11 of Chapter 4, Section 3 in \cite{CLO2007}:

\begin{lemma}\label{l:intersection} Let $I=\langle f_1,\dots,f_k \rangle$ and $J=\langle g_1,\dots,g_l \rangle$ be ideals in $S=\F[x_1,\dots,x_r]$. Then, a Groebner basis of the ideal $ I \cap J$ consists of the polynomials from $S$ in a Groebner basis of the ideal $\langle wf_1,\dots,wf_k,(1-w)g_1,\dots, (1-w)g_l\rangle \subseteq S[w]$ with respect to a lexicographic term order making $w$ the biggest variable.
\end{lemma}

\begin{theorem} \label{t:IdealHirzebruch} Let us fix the following notation:
$$\begin{array}{lllll} &F_1={x}_{3}{x}_{1}f_1={x}_{3}{x}_{1}({x}_{3}^{q-1}-{x}_{1}^{q-1}),\\
&F_2={x}_{4}{x}_{2}{x}_{1}f_2={x}_{4}{x}_{2}{x}_{1}({x}_{4}^{q-1}-{x}_{2}^{q-1}{x}_{1}^{(q-1)\ell}),\\
&F_3={x}_{4}{x}_{3}{x}_{2}f_3={x}_{4}{x}_{3}{x}_{2}({x}_{4}^{q-1}-{x}_{3}^{(q-1)\ell}{x}_{2}^{q-1}),\\
&F_4= {x}_{4}^{q}{x}_{2}-{x}_{4}{x}_{3}^{(q-1)\ell}{x}_{2}^{q}+{x}_{4}{x}_{3}^{q-1}{x}_{2}^{q}{x}_{1}^{(q-1)(\ell-1)}-{x}_{4}{x}_{2}^{q}{x}_{1}^{(q-1)\ell},\\
&F'_4= {x}_{4}^{2q-1}{x}_{2}-{x}_{4}{x}_{3}^{2(q-1)}{x}_{2}^{2q-1}+{x}_{4}{x}_{3}^{q-1}{x}_{2}^{2q-1}{x}_{1}^{q-1}-{x}_{4}{x}_{2}^{2q-1}{x}_{1}^{2(q-1)}.
\end{array}$$
Then, a set of minimal generators for the vanishing ideals are given by:
$$\begin{array}{lllll} 
&I(\mathcal{H}_{\ell}(\F_q))=\langle F_1, F_2, F_3, F_4\rangle=\langle F_1,F_4 \rangle, \text{ if $\ell >1$, and }\\
&I(\mathcal{H}_{1}(\F_q))=\langle F_1, F_2, F_3, F'_4\rangle.
\end{array}$$
\end{theorem}
\begin{proof} Recall from Theorem \ref{t:I(A4)Hirzebruch} that $\mathcal{J}:=I(A^4_G(\F_q))$ is generated by $F_1,F_2,F_3$.

By Corollary \ref{c:saturation}, $I(\mathcal{H}_{\ell}(\F_q))=\mathcal{J}:B$, where $B=\langle x_1x_2,x_1x_4,x_3x_2,x_3x_4\rangle$ and so Proposition 10 of Chapter 4, Section 4 in \cite{CLO2007} implies that $$I(\mathcal{H}_{\ell}(\F_q))=(\mathcal{J}:x_1x_2)\cap (\mathcal{J}:x_1x_4) \cap (\mathcal{J}:x_3x_2) \cap (\mathcal{J}:x_3x_4).$$ 

In the first step, we compute these ideals using the fact that when $\{h_1,\dots,h_k\}$ is a basis for $\cl J \cap \langle g \rangle $ then $\{h_1/g,\dots,h_k/g\}$ is a basis for $\cl J:g$, see Theorem 11 of Chapter 4, Section 4 in \cite{CLO2007}. 

In order to compute a basis for $\cl J \cap \langle x_1x_2 \rangle $, we use Lemma \ref{l:intersection} and compute the Groebner basis of the ideal generated by $wF_1,wF_2,wF_3,(1-w)x_1x_2$ in the ring $S[w]=\F_q[x_1,x_2,x_3,x_4,w]$ with respect to the lexicographic term order with $w>x_4>x_3>x_2>x_1$. It is a routine check that the polynomials 
$$x_2F_1,F_2,(1-w)x_1x_2,wF_1,wF_3$$ 
of $S[w]$ form such a Groebner basis and thus $x_2F_1$ and $F_2 \in S$ generates the ideal $\cl J \cap \langle x_1x_2 \rangle $. Dividing these generators by $x_1x_2$, we get $\cl J : \langle x_1x_2 \rangle =\langle x_3f_1, x_4f_2 \rangle$.

Similarly, the polynomials $x_4F_1,F_2,(1-w)x_1x_4,wF_1,wF_3$ form the Groebner basis of the ideal generated by $wF_1,wF_2,wF_3,(1-w)x_1x_4$ in $S[w]$ with respect to the same term order and thus $x_4F_1/x_1x_4=x_3f_1$ and $F_2/x_1x_4=x_2f_2 \in S$ generates the ideal $\cl J : \langle x_1x_4 \rangle $. 

Once again, the polynomials $x_2F_1,F_3,(1-w)x_3x_2,wF_1,wF_2$ form the Groebner basis of the ideal generated by $wF_1,wF_2,wF_3,(1-w)x_3x_2$ in $S[w]$ and thus $x_2F_1/x_3x_2=x_1f_1$ and $F_3/x_3x_2=x_4f_3 \in S$ generates the ideal $\cl J : \langle x_3x_2 \rangle $.

Finally, the polynomials $x_4F_1,F_3,(1-w)x_3x_4,wF_1,wF_2$ form the Groebner basis of the ideal generated by $wF_1,wF_2,wF_3,(1-w)x_3x_4$ in $S[w]$ and $\cl J : \langle x_3x_4 \rangle $ is generated by $x_4F_1/x_3x_4=x_1f_1$ and $F_3/x_3x_4=x_2f_3 \in S$.

Now, we compute $(\cl J : \langle x_1x_2 \rangle) \cap (\cl J : \langle x_1x_4 \rangle ) $ using Lemma \ref{l:intersection}. The Groebner basis of $\{ wx_3f_1, wx_4f_2,(1-w)x_3f_1,(1-w)x_2f_2 \}$ with respect to the lexicographic term order with $w>x_4>x_3>x_2>x_1$ is computed to be the following set
$\{x_3f_1, x_2x_4f_2,(1-w)x_2f_2,wx_4f_2 \}$ and thus $(\cl J : \langle x_1x_2 \rangle) \cap (\cl J : \langle x_1x_4 \rangle )$ is generated by $x_3f_1$ and $x_2x_4f_2$. Until now, $\ell$ is any positive number. The rest depends on whether $\ell>1$ or $\ell=1$.

\textbf{Case $\ell>1$:} 

As before, the Groebner basis of $\{ wx_3f_1, wx_2x_4f_2,(1-w)x_1f_1,(1-w)x_4f_3 \}$ is computed to be the following set
$\{F_1, F_4,F_5,(w-1)x_1f_1,wx_3f_1,F_6\}$, where
\begin{eqnarray*}
F_4&=&{x}_{4}^{q}{x}_{2}-{x}_{4}{x}_{3}^{(q-1)\ell}{x}_{2}^{q}+{x}_{4}{x}_{3}^{q-1}{x}_{2}^{q}{x}_{1}^{(q-1)(\ell-1)}-{x}_{4}{x}_{2}^{q}{x}_{1}^{(q-1)\ell},\\
F_5&=&{x}_{4}^{q}{x}_{3}^{q}-{x}_{4}^{q}{x}_{3}x_1^{q-1}-x_4 x_3^{(q-1)(\ell+1)+1}{x}_{2}^{q-1}+{x}_{4}{x}_{3}{x}_{2}^{q-1}{x}_{1}^{(q-1)(\ell+1)},\\
F_6&=&(w-1)x_4[{x}_{4}^{q-1}-{x}_{2}^{q-1}x_1^{(q-1)\ell}]+{x}_{4}{x}_{3}^{q-1}{x}_{2}^{q-1}[{x}_{3}^{(q-1)(\ell-1)}-{x}_{1}^{(q-1)(\ell-1)}].
      \end{eqnarray*}
Hence, $\langle x_3f_1, x_2x_4f_2 \rangle \cap (\cl J : \langle x_3x_2 \rangle )$ is generated by $F_1, F_4,F_5$, that is, we obtain $$(\cl J : \langle x_1x_2 \rangle) \cap (\cl J : \langle x_1x_4 \rangle )\cap (\cl J : \langle x_3x_2 \rangle )= \langle F_1,F_4,F_5\rangle.$$

Finally, the Groebner basis of the set $\{ wF_1,wF_4,wF_5, (1-w)x_1f_1,(1-w)x_2f_3\}$ is found to be $\{F_1, F_4,wF_5,(w-1)x_1f_1,(w-1)x_2f_3\}$. Thus, $\langle F_1,F_4,F_5 \rangle \cap (\cl J : \langle x_3x_4 \rangle )$ is generated by $F_1$ and $F_4$, completing the proof for $\ell>1$.

\textbf{Case $\ell=1$:} 

In this case, the Groebner basis of $\{ wx_3f_1, wx_2x_4f_2,(1-w)x_1f_1,(1-w)x_4f_3 \}$ is computed to be the following set
$$\{F_1, F_2,F_3, F'_4,F'_5,(w-1)x_1f_1,wx_3f_1,(w-1)x_4f_3,(w-1)x_2x_4f_3\}, \text{ where}$$
$$
F'_5={x}_{4}^{q}{x}_{3}^{q}-{x}_{4}^{q}{x}_{3}x_1^{q-1}-x_4 x_3^{(q-1)(\ell+1)+1}{x}_{2}^{q-1}+{x}_{4}{x}_{3}{x}_{2}^{q-1}{x}_{1}^{(q-1)(\ell+1)}.
$$
Hence, $\langle x_3f_1, x_2x_4f_2 \rangle \cap (\cl J : \langle x_3x_2 \rangle )$ is generated by $F_1, F_2,F_3, F'_4,F'_5$, that is, we obtain $$(\cl J : \langle x_1x_2 \rangle) \cap (\cl J : \langle x_1x_4 \rangle )\cap (\cl J : \langle x_3x_2 \rangle )= \langle F_1, F_2,F_3, F'_4,F'_5\rangle.$$

Finally, the Groebner basis of the set $$\{ wF_1,wF_2,wF_3,wF'_4,wF'_5, (1-w)x_1f_1,(1-w)x_2f_3\}$$ is found to be $\{F_1, F_2,F_3, F'_4,wF'_5,(w-1)x_1f_1,(w-1)x_2f_3\}$. Thus, we conclude that $\langle F_1, F_2,F_3, F'_4,F'_5 \rangle \cap (\cl J : \langle x_3x_4 \rangle )$ is generated by $F_1, F_2,F_3, F'_4$.
\end{proof}

\begin{remark}
$I(\cl H_{\ell}(\F_q))$ is not binomial, although $I(\A_{G}^4(\F_q))$ is so.
\end{remark}

Generating sets for the vanishing ideals $I(\A^4_G)(\F_q))$ and $I(\cl H_{\ell}(\F_q))$ are found as in the next example. We also illustrate how to find the best choice of a set $Y_G$ between $\cl H_{\ell}(\F_q))$ and $\A^4_G(\F_q))$.

\begin{example} \label{eg:goodcode} 
Let $\beta=\begin{bmatrix}
1 & 0 & 1& \ell\\
0 & 1 & 0& 1  
\end{bmatrix}$ and $q=5$ so that  $\F=\F_5$ and $\K=\overline{\F}_5$. We compute a generating set for the ideal $I(\A^4_G)$, where the group acting on the affine space is
$$G=\Ker(\pi)=\{(t_1,t_2,t_1,t_1^{\ell}t_2)\;|\;t_1,t_2\in\K^*\}\cong (\K^*)^2.$$ 
The following commands computes this vanishing ideal when $\ell=3$:
\begin{verbatim}
i1 : q=5; l=3; F=ZZ/q; beta = matrix {{1,0,1,l},{0,1,0,1}}; 
i2 : r=numColumns beta; d=numRows beta;
i3 : R=F[x_1..x_r,y_1..y_r,z_1..z_d,w];
i4 : f1=y_1,f2=y_2,f3=y_3,f4=y_4;
i5 : J=ideal(x_1-f1*(z_1),x_2-f2*(z_2),x_3-f3*(z_1),
x_4-f4*(z_1)^l*(z_2), y_1^q-y_1,y_2^q-y_2,y_3^q-y_3,y_4^q-y_4,w-1);
i6 : IAG=eliminate (J,for i from r to r+2*d+2 list R_i);
\end{verbatim}
The final output $\verb|IAG|$ is the required ideal:
$$I(\A^4_G)=\langle {x}_{1}^{5}{x}_{3}-{x}_{1}{x}_{3}^{5},{x}_{2}^{5}{x}_{3}^{13}{x}_{4}-{x}_{2}{x}_{3}{x}_{4}^{5},{x}_{1}^{13}{x}_{2}^{5}{x}_{4}-{x}_{1}{x}_{2}{x}_{4
      }^{5}\rangle.$$
In order to compute generators for $I(\cl H_{\ell}(\F_q))$, we use saturation command:
\begin{verbatim}
i7 : S=F[x_1..x_4, Degrees => entries transpose beta];
i8 : IAG=substitute(IAG,S)
i9 : B=ideal(x_1*x_2,x_2*x_3,x_3*x_4,x_4*x_1);
i10 : IX=saturate(IAG,B)
\end{verbatim}
yields \verb|IX| as follows:
$$I(\mathcal{H}_3(\F_5))=\langle {x}_{1}^{5}{x}_{3}-{x}_{1}{x}_{3}^{5},{x}_{1}^{12}{x}_{2}^{5}{x}_{4}-{x}_{1}^{4}{x}_{2}^{5}{x}_{3}^{8}{x}_{4}+{x}_{2}^{5}{x}_{3}^{12}{x}_{4}-{x}_{2}{x}_{4}^{5}\rangle .$$
The difference between $\A^4_G(\F_q)$ and $\cl H_{\ell}(\F_q)$ stems from the following $7$ points:
$$ V(B)=\left\{(0,0,0,0),(0,0,0,1),(0,0,1,0),(0,1,0,0),(0,1,0,1),(1,0,0,0),(1,0,1,0)\right\}$$
Taking $\alpha=(1,0)$, we get $B_{\alpha}=\{x_1,x_3\}$ as a basis for the vector space $(S/I)_{\alpha}$ for $I=I(\cl H_{\ell}(\F_q))$. Adding the three points $Y_3=\left\{[1,0,1,0],[1,0,0,0],[0,0,1,0]\right\}$, will increase the length by $3$. Since a non-zero polynomial $ax_1+bx_3$, for $a,b\in \F_3$, can have at most one extra root among these three points, the minimum distance will increase by two. Indeed, using the Coding Theory package introduced in \cite{CodingTheoryM2}, we compute parameters of the codes $\cl C_{\alpha,Y}$ for $Y=\cl H_{\ell}(\F_q)$ to be $[36,2,30]$ and for $Y=\cl H_{\ell}(\F_q)\cup Y_3$ to be $[39,2,32]$ with the following commands:
\begin{verbatim}
i11 : alpha={1,0}; Bd=flatten entries basis(alpha,coker gens gb IX);
i12 : PX=join(flatten apply(q,i-> apply (q,j-> {i,1,1,j})), 
apply(q,i->{i,0,1,1}), apply(q,i->{1,1,0,i}),{{1, 0, 0, 1}});
i13 : C=evaluationCode(F,PX,Balpha);
[length C.LinearCode, dim C.LinearCode, minimumWeight C.LinearCode]
i14 : PY=join(PX,{{1,0,1,0},{1,0,0,0},{0,0,1,0}});
i15 : C=evaluationCode(F,PY,Balpha);
[length C.LinearCode, dim C.LinearCode, minimumWeight C.LinearCode]
\end{verbatim}
We conclude the example speculating on why the choice we made was the best possible among all $H_{\ell}(\F_q)\subset Y \subset \A^4_G(\F_q)$. Since the weight $w(c_F)$ of a codeword $c_F=(F(P_1),\dots,f(P_{|Y|}))$ is $|Y|-|V_Y(F)|$ it follows that the minimum distance is 
$$\delta({\cl C}_{\alpha,Y})=|Y|-\max \{|V_Y(F)| \,:\, F \in (S/I)_{\alpha} \setminus \{0\}\},$$
where $V_Y(f)=\{[P]\in Y : f(P)=0\}$. Notice that $ax_1+bx_3$ vanishes on the set $Y_0:=V(B)\setminus Y_3$, for every $a,b\in \F_q$. Adding any subset $Y'_0$ of $Y_0$ to a set $Y$ does not increase the length $|Y|$ by $|Y'_0|$ and leaves the minimum distance the same. This is because $|V_Y(F)|$ also increases by the same amount $|Y'_0|$ and so the difference above does not change. Finally, adding a proper subset of $Y_3$ does not increase the minimum distance that much, since for every proper subset of size $1$ there is a polynomial vanishing on that subset. For instance, $x_1-x_3$ vanishes on $\left\{[1,0,1,0]\right\}$. Similarly, no polynomial can have two roots on a proper subset of size $2$ and there is a polynomial with one root, so that the minimum distance increases by $1$. 
\end{example}

\subsection{Weighted Projective Spaces}  Let $w_1,\dots,w_r$ be some positive integers such that $n=r-1$ of them have no nontrivial common divisor, that is we have $\gcd(w_1,\dots,\hat{w}_i,\dots, w_r)=1$, for any $i\in[r]$. In this case, we have a row matrix $\bb=[w_1 \cdots w_r]$ and the corresponding toric variety is denoted $X=\Pp(w_1,\dots,w_r)$. The semigroup $\N\bb$ is the \textit{numerical semigroup} generated by $w_1,\dots,w_r$ denoted also by $\langle w_1,\dots,w_r \rangle$ in the literature. The group $G=\{(t^{w_1},\dots,t^{w_r}) : t\in \K^*\}$ is the torus of the affine monomial curve parameterized by $x_i=t^{w_i}$, where $i\in [r]$. The toric ideal $I_{L_{\bb}}$ is the defining ideal of this monomial curve whose coordinate ring is the semigroup ring $\K[\N\bb]=\K[t^{w_1},\dots,t^{w_r}]$ when $\K=\overline{\F}_q$. 

\begin{proposition} \label{p:IdealWPS}
If $X=\Pp(w_1,\dots,w_r)$ is the weighted projective space, then its vanishing ideal $I(X(\F_q))=I(\A^r_G(\F_q))$.
\end{proposition}
\begin{proof} As $\A^r_G(\F_q)=X(\F_q) \cup \{0\}$, we have the following equalities \\ $I(\A^r_G(\F_q))=I(X(\F_q)) \cap I(\{0\})=I(X(\F_q)) \cap \langle x_1,\dots,x_r \rangle=I(X(\F_q))$.
\end{proof}
If $w_i=1$, for all $i\in [r-2]$, but $w_{r-1}=a$ and $w_{r-1}=b$ are arbitrary, the vanishing ideal $I(X(\F_q))$ for $X=\Pp(1,\dots,1,a,b)$ is easy to compute.
\begin{theorem} \label{t:IdealWPS} For the weighted projective space $X=\Pp(1,\dots,1,a,b)$, the vanishing ideal $I(X(\F_q))$ is generated by the following binomials
$$\begin{array}{lllll} 
&x_ix_j(x_i^{q-1}-x_j^{q-1}) & \text{ for } 1\leq i < j <r-1, \\
&x_kx_{r-1}(x_k^{(q-1)a}-x_{r-1}^{q-1}) &\text{ for } 1\leq k <r-1, \\  
&x_kx_{r}(x_k^{(q-1)b}-x_r^{q-1}) &\text{ for } 1\leq k <r-1, \\
& x_{r-1}x_{r}(x_{r-1}^{(q-1)b}-x_r^{(q-1)a}).
\end{array}$$
\end{theorem}
\begin{proof} By the virtue of Proposition \ref{p:IdealWPS}, it suffices to find generators for the ideal $I(\A^4_G)(\F_q)$ which by Theorem \ref{t:idealAffine} come from $\x^{\varepsilon}I_{(q-1)L_{\bb(\varepsilon)}}$. When $|\varepsilon|<2$, the toric ideal of the numerical semigroup corresponding to $\bb(\varepsilon)$ is trivial. When $|\varepsilon|=2$, the toric ideal $I_{(q-1)L_{\bb(\varepsilon)}}$ is a complete intersection generated by one of the binomials below: 
$$\begin{array}{lllll} 
&f_{i,j}=x_i^{q-1}-x_j^{q-1} &\text{ if } \varepsilon=\{i,j\} \text{ for } 1\leq i < j <r-1, \\
&f_{k,r-1}=x_k^{(q-1)a}-x_{r-1}^{q-1} &\text{ if } \varepsilon=\{k,r-1\} \text{ for } 1\leq k <r-1, \\  
&f_{k,r}=x_k^{(q-1)b}-x_r^{q-1} & \text{ if } \varepsilon=\{k,r\} \text{ for } 1\leq k <r-1, \\
&f_{r-1,r}=x_{r-1}^{(q-1)b}-x_r^{(q-1)a} & \text{ if } \varepsilon=\{r-1,r\}.
\end{array}$$
Therefore, the generators coming from $\x^{\varepsilon}I_{(q-1)L_{\bb(\varepsilon)}}$ are exactly the binomials given in the statement of the Theorem \ref{t:IdealWPS}. Now, we prove that they are indeed sufficient, since when $|\varepsilon|>2$ they divide the rest of the binomials. For if $\varepsilon=\{i_1,\dots,i_k\}$, then $I_{(q-1)L_{\bb(\varepsilon)}}$ is a complete intersection generated by $k-1$ of the binomials $f_{i,j}$, $f_{k,r-1}$, $f_{k,r}$ and $f_{r-1,r}$ above. Thus, the generators coming from $\x^{\varepsilon}I_{(q-1)L_{\bb(\varepsilon)}}$ will be the $k-1$ of the binomials $x_{i_1}\cdots x_{i_k}f_{i,j}$, $x_{i_1}\cdots x_{i_k}f_{k,r-1}$, $x_{i_1}\cdots x_{i_k}f_{k,r}$ and $x_{i_1}\cdots x_{i_k}f_{r-1,r}$ which are divisible by the binomials coming from the case $|\varepsilon|=2$.
\end{proof}
As a particular case we single out the following.
\begin{corollary}
$I(\Pp(1,a,b)(\F_q))$ is generated by the following binomials
$$x_1x_{2}(x_1^{(q-1)a}-x_{2}^{q-1}), \quad x_1x_{3}(x_1^{(q-1)b}-x_3^{q-1}), \quad x_{2}x_{3}(x_{2}^{(q-1)b}-x_3^{(q-1)a}).$$
\end{corollary}
\begin{proof}
Direct consequence of Theorem \ref{t:IdealWPS}.
\end{proof}
\begin{remark} Mercier and Rolland
\cite{MRidealOfPn} has given a binomial generating set for the ideal $I(\Pp^n(\F_q))$ and Theorem \ref{t:IdealWPS} generalizes this result to some weighted projective spaces. We recommend the paper \cite{BDGidealOfPn} by Beelen, Datta and Ghorpade in order to see how they use the set given by \cite{MRidealOfPn} to obtain a footprint bound for the minimum distance of the corresponding code.
\end{remark}
One can use the vast literature about numerical semigroups and their toric ideals together with Theorem \ref{t:idealAffine} and Proposition \ref{p:IdealWPS} to give generating sets for families of weighted projective spaces. In order to state some of the results scattered the literature we recall some key concepts. For a numerical semigroup $W$ generated by $w_1,\dots,w_r$, the subset of pseudo-Frobenius numbers are defined by
$$PF(W)=\{ z\in\Z \setminus W:z+w\in W \; \mbox{for all} \; w\in W\setminus \{0\}\}.$$ The largest integer $g(W)\notin W$ belongs to $PF(W)$ and is called
the Frobenius number of $W$. If $PF(W)=\{ g(W)\}$, then $W$ is called symmetric, whereas if $PF(W)=\{ g(W)/2,g(W)\}$, it is called pseudosymmetric.

It is well known that any of  $\Pp(lw_1,lw_2,w_3)$, $\Pp(lw_1,w_2,lw_3)$ or $\Pp(w_1,lw_2,lw_3)$ is isomorphic to $\Pp(w_1,w_2,w_3)$, for any positive integer $l$, we assume that $w_1,w_2$ and $w_3$ are relatively prime to each other and $w_1<w_2<w_3$.
\begin{proposition} If $W$ is symmetric, then $w_3=a_{31}w_1+a_{32}w_2$ for some non-negative integers $a_{31}$ and $a_{32}$ and the vanishing ideal of $\Pp(w_1,w_2,w_3)(\F_q)$ is generated by the following $4$ binomials
$$
\begin{array}{lll}
     &  x_1x_2(x_1^{(q-1)w_2}-x_2^{(q-1)w_1}),  \quad &x_1x_3(x_1^{(q-1)w_3}-x_3^{(q-1)w_1}),\\
     & x_2x_3(x_2^{(q-1)w_3}-x_3^{(q-1)w_2}),  \quad
&x_1x_2x_3(x_3^{q-1}-x_1^{(q-1)a_{31}}x_2^{(q-1)a_{32}}).
\end{array}
$$
If $W$ is not symmetric, then there are $a_1,a_2$ and $a_3$ such that $a_iw_i=a_{ij}w_j+a_{ik}w_k$, for $\{ i,j,k\}=\{ 1,2,3\}$ and the vanishing ideal of $\Pp(w_1,w_2,w_3)(\F_q)$ is generated by the following $6$ binomials
$$
\begin{array}{lll}
&  x_1x_2(x_1^{(q-1)w_2}-x_2^{(q-1)w_1}),  \quad 
&x_1x_2x_3(x_1^{(q-1)a_1}-x_2^{(q-1)a_{12}}x_3^{(q-1)a_{13}}),\\
& x_1x_3(x_1^{(q-1)w_3}-x_3^{(q-1)w_1}),\quad
& x_1x_2x_3(x_2^{(q-1)a_2}-x_1^{(q-1)a_{21}}x_3^{(q-1)a_{23}}),\\
&x_2x_3(x_2^{(q-1)w_3}-x_3^{(q-1)w_2}), \quad 
 &x_1x_2x_3(x_3^{(q-1)a_3}-x_1^{(q-1)a_{31}}x_2^{(q-1)a_{32}}).
\end{array}
$$
\end{proposition}
\begin{proof}
If $W$ is symmetric, then by \cite[Theorem 3.10]{JH1970}, $w_3=a_{31}w_1+a_{32}w_2$ for some non-negative integers $a_{31}$ and $a_{32}$, and the toric ideal of the semigroup $W$ is generated by  $x_1^{w_2}-x_2^{w_1}$ and $x_3-x_1^{a_{31}}x_2^{a_{32}}$. When $\varepsilon=\{1,2,3\}$, $\N\beta(\varepsilon)=W$, so we get the  binomials $x_1x_2x_3(x_1^{(q-1)w_2}-x_2^{(q-1)w_1})$ and $x_1x_2x_3(x_3^{q-1}-x_1^{(q-1)a_{31}}x_2^{(q-1)a_{32}})$ from here. If $\varepsilon=\{1,2\}$, then $\N\beta(\varepsilon)=\langle w_1,w_2 \rangle$, and so we get the binomial $x_1x_2(x_1^{(q-1)w_2}-x_2^{(q-1)w_1})$. Similarly, $\varepsilon=\{1,3\}$ gives $\N\beta(\varepsilon)=\langle w_1,w_3 \rangle$ and the binomial $x_1x_3(x_1^{(q-1)w_3}-x_3^{(q-1)w_1})$ and finally $\varepsilon=\{2,3\}$ gives the binomial $x_2x_3(x_2^{(q-1)w_3}-x_3^{(q-1)w_2})$, completing the proof for the first case.

If $W$ is not symmetric, then by \cite[Proposition 3.2]{JH1970} there are positive integers $a_1,a_2$ and $a_3$ such that $a_iw_i=a_{ij}w_j+a_{ik}w_k$, for $\{ i,j,k\}=\{ 1,2,3\}$, satisfying
$a_{21}+a_{31}=a_1,a_{12}+a_{32}=a_2,a_{13}+a_{23}=a_3$, and the toric ideal is generated by
$$g_1=x_1^{a_1}-x_2^{a_{12}}x_3^{a_{13}}, \quad g_2=x_2^{a_2}-x_1^{a_{21}}x_3^{a_{23}}, \quad g_3=x_3^{a_3}-x_1^{a_{31}}x_2^{a_{32}}.$$ In fact, these $a_i$'s are the smallest positive integers with that property. Thus, when
$\varepsilon=\{1,2,3\}$, $\N\beta(\varepsilon)=W$, so we get the generators 
$$x_1x_2x_3g_1(x_1^{q-1},x_2^{q-1},x_2^{q-1}),\, x_1x_2x_3g_2(x_1^{q-1},x_2^{q-1},x_2^{q-1}), \, x_1x_2x_3g_3(x_1^{q-1},x_2^{q-1},x_2^{q-1}).$$ 
If $\varepsilon=\{1,2\}$, then $\N\beta(\varepsilon)=\langle w_1,w_2 \rangle$, and so we get $x_1x_2(x_1^{(q-1)w_2}-x_2^{(q-1)w_1})$ as in the first case. Similarly, $\varepsilon=\{1,3\}$ gives $\N\beta(\varepsilon)=\langle w_1,w_3 \rangle$ and the binomial $x_1x_3(x_1^{(q-1)w_3}-x_3^{(q-1)w_1})$ and finally $\varepsilon=\{2,3\}$ gives $x_2x_3(x_2^{(q-1)w_3}-x_3^{(q-1)w_2})$, completing the proof for the second case.
\end{proof}

\begin{remark}
Let $X=\Pp(1,1,2)$ and $\K=\overline{\F}_3$. Then, the $\F_3$-rational points are $X(\F_3)=(\F_3^3\setminus \{0\})/G$, where $G=\{(\lambda,\lambda,\lambda^2) : \lambda \in \K^*\}$. However, we can not replace $G$ by the subgroup $G(\F_3)=\{(\lambda,\lambda,\lambda^2) : \lambda \in \F_3^*\}$. For instance, the points $[0:0:1]$ and $[0:0:2]$ are the same in $X(\F_3)$, as there is a $\lambda\in \K^*$ with $\lambda^2=2$ so that $(\lambda,\lambda,\lambda^2)\cdot(0,0,1)=(0,0,2)$. But for any $\lambda \in \F_3^*$, $\lambda^2=1$ and $[0:0:1] \neq [0:0:2]$ in $(\F_3^3\setminus \{0\})/G(\F_3)$. However, these points have the same vanishing ideal $\langle x_1,x_2\rangle$ in $S=\F_3[x_1,x_2,x_3]$ in any case.
\end{remark}

\subsection{Product of Projective Spaces}
The product of projective spaces is also a toric variety denoted by $X=\Pp^{n_1}\times \cdots \times \Pp^{n_k}$ with the class group isomorphic to $\Z^k$. The Cox ring $S=\F_q[x_{1,1},\dots,x_{1,r_1},\dots,x_{k,1},\dots,x_{k,r_k}]$ is graded via 
$$\deg(x_{1,1})=\cdots=\deg(x_{1,r_1})=\e_1, \dots,\deg(x_{k,1})=\cdots=\deg(x_{k,r_k})=\e_k,$$
where $\e_1,\dots,\e_k\in \Z^k$ form the standard basis, and $r_i=n_i+1$, for $i\in [k]$. The monomial ideal is $$B=\langle x_{1,1},\dots,x_{1,r_1}\rangle \cap \cdots \cap \langle x_{k,1},\dots,x_{k,r_k}\rangle.$$
\begin{corollary}
If $X=\Pp^{n_1}\times \cdots \times \Pp^{n_k}$ is a product of projective spaces then $I(X(\F_q))=I(\A^r_G(\F_q))$.
\end{corollary}
\begin{proof} Recall that $X=\A^r_G \setminus V_G(B)$. Since $X(\F_q)$ and $\A^r_G(\F_q)$ are finite, their ideals are given by $$\displaystyle I(X(\F_q))=\bigcap_{[P]\in X(\F_q)} I([P]) \quad \text{ and } I(\A^r_G(\F_q))=\bigcap_{[P]\in \A^r_G(\F_q)} I([P]).$$
Our aim is to prove that for any $[P]\in \A^r_G(\F_q)$ there is a point $[P']\in X$ with $I([P'])\subset I([P])$ so the intersections are the same. If $[P]\in X$, then $[P']=[P]$. If $[P]\in V_G(B)$ with support $\varepsilon$, then $[P]\in V_G(x_{i_0,1},\dots,x_{i_0,r_{i_0}})$ for some $i_0\in [k]$. Then, we define the point $P'=(p'_{i,j})$ with support $\varepsilon'=\varepsilon \cup \{(i_0,1)\}$ in such a way that $p'_{i,j}=p_{i,j}$ for $(i,j)\in \varepsilon$ and $p'_{i_0,1}=1$. Then, clearly, $\mathfrak{m}(\widehat{\varepsilon'})\subset \mathfrak{m}(\hat{\varepsilon})$ and $x_{i_0,1}\in \mathfrak{m}(\hat{\varepsilon})\setminus \mathfrak{m}(\widehat{\varepsilon'})$. Since $(i_0,j)\notin \varepsilon$, for all $j\in {r_{i_0}}$, it follows that $L_{\bb(\varepsilon')}=L_{\bb(\varepsilon)}\times \{0\}$ and $\kay_p'(\m,0)=\kay_p(\m)$ thus $I_{\kay_p',L_{\bb(\varepsilon')}}=I_{\kay_p,L_{\bb(\varepsilon)}}$. 

By Proposition \ref{p:IdealPoint}, we have $I([P])=\mathfrak{m} (\check{\varepsilon})+S\cdot I_{\kay_p,L_{\bb(\varepsilon)}}$. Therefore, $I([P'])\subset I([P])$. If we still have $[P']\in V_G(B)$, then the same procedure will give the chain $I([P''])\subset I([P'])\subset I([P])$ and continuing this way if necessary we end up with the desired point in $X$.
\end{proof}
\begin{example} Let $\beta=\begin{bmatrix}
1 & 1& 1& 0 & 0 &0 & 0 \\
0 & 0& 0& 1 & 1 &1 & 1 
\end{bmatrix}$ and $q=3$ so that  $\F=\F_3$ and $\K=\overline{\F}_3$. Our toric variety is $X=\Pp^2 \times \Pp^3 $ and its Cox ring is $S=\F[x_1,\dots,x_7]$ graded via:
$$
\begin{array}{llll}
\deg_{\beta}(x_1)=\deg_{\beta}(x_2)=\deg_{\beta}(x_3)=(1,0);\\ \deg_{\beta}(x_4)=\deg_{\beta}(x_5)=\deg_{\beta}(x_6)=\deg_{\beta}(x_7)=(0,1).
\end{array}
$$
We compute a generating set for the vanishing ideal $I(X(\F_3))=I(\A^7_G(\F_3))$ with the following commands:
\begin{verbatim}
i1 : q=3; F = GF(q,Variable => a); 
beta = matrix {{1,1,1,0,0,0,0},{0,0,0,1,1,1,1}}; 
i2 : r=numColumns beta; d=numRows beta;
i3 : R=F[x_1..x_r,y_1..y_r,z_1..z_d];
i4 : f1=y_1,f2=y_2,f3=y_3,f4=y_4,f5=y_5,f6=y_6,f7=y_7;
i5 : J=ideal(x_1-f1*(z_1),x_2-f2*(z_1),x_3-f3*(z_1),x_4-f4*(z_2),
x_5-f5*(z_2),x_6-f6*(z_2),x_7-f7*(z_2),y_1^q-y_1,y_2^q-y_2,
y_3^q-y_3,y_4^q-y_4,y_5^q-y_5,y_6^q-y_6,y_7^q-y_7)
i6 : IAG=eliminate (J,for i from r to d+2*r-1 list R_i)
\end{verbatim}
The final output $\verb|IAG|$ is the required ideal:
$$
\begin{array}{lll}
I(\A^7_G)=\langle {x}_{6}^{3}{x}_{7}-{x}_{6}{x}_{7}^{3},{x}_{5}^{3}{x}_{7}-{x}_{5}{x}_{7}^{3},{x}_{4}^{3}{x}_{7}-{x}_{4}{x}_{7}^{3},{x}_{5}^{3}{x}_{6}-{x}_{5}{x}_{6}^{3},\\ 
{x}_{4}^{3}{x}_{6}-{x}_{4}{x}_{6}^{3},{x}_{4}^{3}{x}_{5}-{x}_{4}{x}_{5
      }^{3},{x}_{2}^{3}{x}_{3}-{x}_{2}{x}_{3}^{3},{x}_{1}^{3}{x}_{3}-{x}_{1}{x}_{3}^{3},{x}_{1}^{3}{x}_{2}-{x}_{1}{x}_{2
      }^{3}\rangle.
\end{array}
$$
\end{example}

\subsection{A combinatorial method to compute the dimension} \label{s:dimension}
In this section we assume $X=X_{\Sigma}$ is a simplicial complete (not necessarily projective) toric variety. Let $D=\sum_{i=1}^r a_iD_i$ be an \textit{ample} divisor on $X$ of degree $\aa=\sum_{i=1}^r a_i\bb_i $, where $D_i=V(x_i)$. Then, the polytope $$P_D=\{\uu \in \Z^n : \la \uu, \vv_i \ra \geq -a_i , \forall i \in [r]\}$$  is \textit{ample}, that is, its normal fan is $\Sigma$. So, $P_D$ is also a full dimensional lattice polytope having a unique facet representation 
\[\displaystyle P_D=\bigcap_{i=1}^r H^{+}_{i,D}, \text{ where } H^{+}_{i,D}=\{\uu \in \Z^n : \la \uu, \vv_i \ra \geq -a_i  \}
\]
with a supporting hyperplane $H_{i,D}=\{\uu \in \Z^n : \la \uu, \vv_i \ra + a_i=0  \}$. The facets of $P_D$ are given by $F_{i,D}=\{\uu \in P_D : \la \uu, \vv_i \ra + a_i=0  \}$ for $i \in [r]$.

Proper faces $Q_D$ of $P_D$ are the intersection of facets containing it, i.e. 
 \begin{equation} \label{eq:faces}
Q_D= \bigcap_{Q_{D} \subseteq F_{i,D}} F_{i,D}=  \bigcap_{i\in \vep^c } F_{i,D} \text{ for } \vep^c:=[r]\setminus \vep= \{ i\in [r] : Q_{D} \subseteq F_{i,D}\}.
 \end{equation}
Therefore, there is a bijection between the faces $Q_{D}$ of $P_D$ and the complements $\vep$ of the subsets $\{ i\in [r] : Q_{D} \subseteq F_{i,D}\}$, and $P_D$ correspond to $\vep=[r]$. 

Recall that faces $Q$ of a polytope $P$ are denoted by $Q\prec P$ and its interior consists of points not lying on any of its proper faces, i.e.

\begin{equation*}\label{eq:interior}
P^{\circ}=P \setminus \bigcup_{\substack{Q\prec P \\Q \neq P}} Q.
\end{equation*}

\begin{definition}\cite[Definition 3.4]{NardiProjToric} An equivalence relation $\sim_{P}$ on the set of lattice points $P\cap \Z^{n}$ is defined by
\[\uu \sim_{P} \uu' \iff \exists Q\prec P \mbox{ such that } \uu,\uu' \in Q^{\circ} \mbox{ and } \uu-\uu'^{} \in (q-1)\Z^{n}\] where $Q^{\circ}$ is the interior of $Q$.
A \textbf{projective reduction} $\red{P}$ of $P$ is defined to be a set of representatives of elements of $P\cap \Z^{N}$ modulo $\sim_{P}.$
\end{definition}
There is a well known $1-1$ correspondence between the lattice points of $P_D$ and a basis of the vector space $S_{\aa}$, via 
\begin{equation*}
    \displaystyle \uu\in P_D \cap \Z^n \rightarrow \kay^{\la \uu,P_D\ra}=\x^{\m}=\prod_{i=1}^r x_i^{\la \uu, \vv_i \ra + a_i}\in S_{\aa}, \text{ where } m_i=\la \uu, \vv_i \ra + a_i.
\end{equation*}
We use the following in the sequel.
\begin{lemma}
If $\aa$ is an ample degree then $I_{\aa}(X(\F_q))=I_{\aa}(\A^r_G({\F_q}))$.   
\end{lemma}
\begin{proof}
As $X(\F_q)\subseteq \A^r_G({\F_q})$, we need only to prove that $I_{\aa}(X(\F_q))\subseteq I_{\aa}(\A^r_G({\F_q}))$. This will be done once we prove that $S_{\aa} \subset B$, since in that case $F\in I_{\aa}(X(\F_q)) \subset S_{\aa}$ will be an element of $B$ vanishing also on $V_G(B)=\A^r_G\setminus X$. 

If $\uu\in P^{\circ}_D\cap \Z^n$, then $\x^{\hat{\sigma}}$ divides $x_1\cdots x_r$ which divides $\kay^{\la \uu,P_D\ra}$ for any $\sigma \in \Sigma_{P_D}$ implying that $\kay^{\la \uu,P_D\ra}\in B$. If $\uu\in Q_D\cap \Z^n$ for a proper face $Q_D$, then there is a cone $\sigma\in \Sigma_{P_D}$ spanned by the inner normal vectors $\vv_{i_1},\dots,\vv_{i_k}$ of $Q_D$ as $P_D$ is ample and $\la \uu, \vv_i \ra + a_i=0 \iff i \in \{i_1,\dots,i_k\}$. Thus, $\x^{\hat{\sigma}}$ divides $\kay^{\la \uu,P_D\ra}$ implying that $\kay^{\la \uu,P_D\ra}\in B$.    
\end{proof}
Next, we give an algebraic proof for \cite[Theorem 3.5]{NardiProjToric} which is a very useful combinatorial method for computing the dimension of the code obtained from $X(\F_q)$.
\begin{theorem}\label{t:dimension}
If $\aa$ is ample, a basis for the code $C_{\aa,Y}$ on $Y=X(\F_q)$ is given by the images, under the evaluation map $\ev_Y$, of monomials $\kay^{\la \uu,P_D\ra}$ where $\uu\in \red(P_D)$. Therefore $K=\dim_{\F_q} C_{\aa,Y}=|\red(P_D)|$. 
\end{theorem}
\begin{proof} We show that  $H_Y({\aa})=|\red_{\succ}(P_D)|$ for the projective reduction $\red_{\succ}{P}$ whose elements correspond to monomials that are the biggest with respect to a term order $\succ$. Indeed, this will follow from the assertion that $I_{\aa}(Y)=I_{\aa}(\A^r_G({\F_q}))$ and
 \begin{equation}
     \kay^{\la \uu',P_D\ra} -\kay^{\la \uu{''},P_D\ra} \in I_{\aa}(Y)  \iff \uu{'} \sim_{P_D} \uu{''},
 \end{equation}
 since the ideal $I(A^r_G({\F_q}))$ is binomial. 
 
 Before going further let us set $\supp( \kay^{\la \uu,P_D\ra}):=\{i\in [r]: \la \uu, \vv_i \ra + a_i >0\}$. If $\vep= \supp( \kay^{\la \uu{'},P_D\ra})\cap \supp( \kay^{\la \uu{''},P_D\ra})$, then we have 
 \begin{equation} \label{eq:binomial}
 \displaystyle \kay^{\la \uu{'},P_D\ra} -\kay^{\la \uu{''},P_D\ra}=\x^{\m^{'}}-\x^{\m^{''}}=\prod_{i\in \vep} x_i^{\la \uu^{}, \vv_i \ra + a_i}  (\x^{\m^+}-\x^{\m^-})  
 \end{equation}
 where $\m^{+},\m^- \in \N^r$ satisfying $\m^{+}-\m^{-}=\m=\m^{'}-\m^{''}\in \Z^r$. 
 
 Now, if $ \kay^{\la \uu',P_D\ra} -\kay^{\la \uu{''},P_D\ra} \in I_{\aa}(Y) =I_{\aa}(\A^r_G(\F_q))$, then by the proof of Theorem \ref{t:idealAffine}, it follows that $\mbox{supp}(\x^{\m^+})\cup \mbox{supp}(\x^{\m^-})\subseteq\varepsilon$ yielding $ \kay^{\la \uu',P_D\ra} -\kay^{\la \uu{''},P_D\ra} \in I_{\aa}(\A^r_G(\vep)) $. Hence, by Theorem \ref{t:idealOrbit}, we get $\m^{+}-\m^{-}\in (q-1)L_{\bb}(\vep)$ and $\uu^{'},\uu^{''} \in Q^{\circ}_D$, for the face $\displaystyle Q_D=  \bigcap_{i\in \vep^c } F_{i,D}$ of $P_D$ described in \eqref{eq:faces} corresponding to $\vep$. As we clearly have $ \uu^{'}-\uu^{''} \in (q-1)\Z^n$, it follows that $\uu{'} \sim_{P_D} \uu{''}$.

 Conversely, if $\uu{'} \sim_{P_D} \uu{''}$ then there is a face $Q_D$ of $P_D$ whose interior contains both $\uu^{'}$ and $\uu^{''}$ with $ \uu^{'}-\uu^{''} \in (q-1)\Z^n$. Again as in \eqref{eq:faces}, we write $\displaystyle Q_D=  \bigcap_{i\in \vep^c } F_{i,D}$ for $\vep^c=\{ i\in [r] : Q_{D} \subseteq F_{i,D}\}$. Observe now that if $\uu^{}\in Q^{\circ}_D$, no other face $F_{j,D}$ can contain $\uu^{}$ for any $j\in \vep$. Hence, $\la \uu^{}, \vv_j \ra + a_j >0$ or equivalently $x_j$ divides $\kay^{\la \uu{},P_D\ra}$ for any $j\in \vep$. Thus, it follows that $ \supp( \kay^{\la \uu{'},P_D\ra})= \supp( \kay^{\la \uu{''},P_D\ra})=\vep$. Notice that $\x^{\vep}$ divides both terms of the binomial in \eqref{eq:binomial} and $\x^{\m^+}-\x^{\m^-}\in I_{(q-1)L_{\bb(\vep)}}$. As in the proof of of Theorem \ref{t:idealAffine}, we also have that $\x^{\vep}(\x^{\m^+}-\x^{\m^-}) $ vanishes on $\A^r_G(\F_q)$. Therefore, $\kay^{\la \uu{'},P_D\ra} -\kay^{\la \uu{''},P_D\ra} \in I_{\aa}(\A^r_G(\vep)) $, completing the proof.
\end{proof}

\section*{Acknowledgements} The author would like to thank Jade Nardi for useful conversations especially on the last section.

\bibliographystyle{plain}      
\bibliography{SahinComputingVanishingIdealsForToricCodes}   
\end{document}